\documentclass[11pt,reqno,a4paper]{amsart}
\usepackage[english]{babel}

\usepackage{amsmath,amsthm,amsfonts,amssymb}
\usepackage{mathtools}
\usepackage{color}
\usepackage{csquotes}

\usepackage{mathrsfs}
\usepackage[all]{xy}
\usepackage{tikz-cd}
\usepackage{datetime}

\usepackage{hyperref}

\theoremstyle{definition}

\usepackage{graphicx}
\usepackage{tikz}
\usepackage{biblatex}
\usetikzlibrary{positioning}

\newtheorem{theorem}{Theorem}[section]

\newtheorem{lemma}{Lemma}[section]
\newtheorem{definition}{Definition}[section]
\newtheorem{remark}{Remark}[section]

\title{A Model for the Multi-Virus Contact Process}

\author[Xu Huang]{Xu Huang}
\address{Xu Huang: Dept. of Mathematics
\\University of Rochester
\\Rochester, NY  14627}
\email{xhuang46@u.rochester.edu}

\date{\today}

\begin{document}


\begin{abstract}
    We study one specific version of the contact process on a graph. Here, we allow multiple infections carried by the nodes and include a probability of removing nodes in a graph. The removal probability is purely determined by the number of infections the node carries at the moment when it gets another infection. In this paper, we show that on any finite graph, any positive value of infection rate $\lambda$ will result in the death of the process almost surely. In the case of $d$-regular infinite trees, We also give a lower bound on the infection rate in order for the process to survive, and an upper bound for the process to die out.
\end{abstract}



\maketitle
\section{Introduction}

\maketitle
\subsection{Background}
The contact process, first introduced by T. E. Harris in \cite{T.E.Harris}, is a model describing the spread of infections on a graph. A general description of this model can be found in \cite{grimmett2018probability}. There are many beautiful papers in this area, such as \cite{nam2022critical} and \cite{durrett1988lecture}. Because of the COVID-19, people are interested in getting meaningful results regarding the pandemic through the methodology of math modeling, and papers like \cite{dauvergne2022sir} and \cite{candellero2021first} give us some examples of fascinating variants of contact processes. In this paper, instead of working with the traditional SIS (Susceptible - Infected - Susceptible) contact process, we reformulate the process by introducing a probability of death for each node. In our process, there are three statuses for each node once the virus appears in the graph: the node can get infected and affect others, it can be healthy, or it can be removed from the graph. Our model is called the Multi-Virus Contact Process (MVCP) with death rate acting on each node of the graph. Here is the verbal description of the model: 
\begin{enumerate}

\item Each vertex in the graph is allowed to carry multiple infections at any given moment.
\item Each infection passes to its host's neighbors independently following a Poisson process with infection rate $\lambda$.
\item Each infection is healed following a Poisson process with rate 1.
\item Each infection or recovery happens independently in the process.
\item Each infection has a probability of killing its host at the moment it passes to the host. When the infection passes to the host, suppose the host is already carrying $i$ infections simultaneously, the probability of killing the host is $\phi(i+1)$, where $i \in \mathbb{N}$. $\phi$ is the function which maps the number of infections the host is carrying to the probability of death, at the moment the infection passes to the host. 
\item When a node dies, the node itself, all the infections on it, and all the edges connecting to it are removed from the graph.
\end{enumerate}

There are two main differences between the classical contact process and our new model. The first is that we allow any node to carry multiple infections, and the second is that we add the probability of removing nodes and edges from the graph by including the probability of death. 

Also, due to the fact that the probability of death only depends on the current number of viruses on the node, our model is related to the zero-range process. First introduced by Frank Spitzer in \cite{SPITZER1970246}, the zero-range process (ZRP) is a type of stochastic interacting particle system in which indistinguishable particles can relocate from site to site, with rates depending on the occupancy of the departure sites. Then one can view the viruses in the multi-virus contact process model as particles giving birth, dying, and moving among the nodes. The death of any node corresponds to the removal of a site from the zero-range process. The probability of removing such a site is purely determined by the number of particles (or viruses) on it. This paper does not use methods from ZRP, but the reader interested in such methods can consult \cite{10.1214/aop/1176996977} and \cite{liggett1985interacting}. 

Despite the fact that contact processes have been investigated extensively in the past decades, we are interested in this particular variant of contact process because we believe this model can be useful to describe the situation in a variety of fields, such as social science and epidemiology. For instance, if we treat rumor as the virus, so that each person in a social circle, which is represented by the node in the graph, can receive the rumor (gets infected) or forget (heals) it. The person always hears the rumor from the people around him/her, and he/she also has the ability to pass the rumor to others, represented by the spread of virus in the model. Moreover, if the person has heard enough of the rumor, he/she can choose to leave the social circle by cutting out the connections with others in the circle, which is represented by the removal of the node in the graph. The model can also describe the word of mouth in social media, or the spread of bacteria in a community, and we will omit the explanation.

Note the function of death $\phi$ only depends on the number of infections a node currently carries, without taking time as another parameter. This setup indicates that we assume the illness of a node only worsen at the moment the node gets a new infection. The reason of not considering time in the function is that we assume the patient is diagnosed as soon as he/she is infected, and that the patient is going to heal from the infections if he/she is diagnosed to live by the doctor, represented by the function $\phi$.

\maketitle
\subsection{Mathematical Formulation}
Here is the mathematical formulation of our contact model. On a given graph $G_0 = (V_0, E_0)$  with vertex set $V_0$ and edge set $E_0$ at the moment $t=0$, our contact process $(\xi_t)_{t \geq 0}$ is a continuous Markov process with infection rate $\lambda$ and recovery rate 1 on the state space $\{\mathbb{N} \cup \emptyset\}^{V_0}$, where $\mathbb{N}$, from now on, is the set of natural numbers including 0. Here, on a given site $x$, we write $\xi_t (x) = 0$ if the node is healthy, $\xi_t (x) = i$ if it carries $i$ infections and $\xi_t (x) = \emptyset$ if it dies or is already dead, at moment $t$. If $\xi_t (x) = \emptyset \text{ or } 0\quad \forall x \in V_0$, then of course there do not exist any infected nodes on the graph at the moment $t$ anymore, and we denote this case by $\xi_t = \emptyset$. Therefore, assuming that the state of the contact process is $\zeta$ at time $t$, and $\zeta\ne\emptyset$, then we have 
\begin{equation*}
\begin{split}
& \mathbb{P} (\xi_{t+h} (x) = \zeta(x) -1 \mid \xi_{t} = \zeta) =  \zeta(x)h + o(h)  \qquad  \qquad \qquad  \text{if} \hspace{0.1cm} \zeta (x) \ne 0  \hspace{0.1cm} \text{or} \hspace{0.1cm}  \emptyset \\
&\mathbb{P} (\xi_{t+h} (x) = \zeta(x)+1 \mid \xi_{t} = \zeta) = (\lambda N^t_{\zeta}(x) h + o(h))(1-\phi(\zeta(x)+1)) \hspace{0.45cm}  \text{and}    \\
&\mathbb{P} (\xi_{t+h} (x) = \emptyset \mid \xi_{t} = \zeta) = (\lambda N^t_{\zeta}(x) h + o(h))\phi(\zeta(x)+1)  \hspace{0.37cm}\quad \quad \quad  \text{if} \hspace{0.1cm} \zeta (x) \ne  \emptyset 
\end{split}
\end{equation*} 
as $h \downarrow 0$. $N^t_{\zeta}(x)$ is the number of infections which the neighbours of $x$ have in $\xi$ at moment $t$,

\begin{equation*}
N^t_{\zeta}(x) = \sum_{y \in \space V_0, y \sim_t x} {\zeta(y)},
\end{equation*} 
where $ x \sim_t y$ means that there still exists an edge connecting $x$ and $y$ at moment $t$. $\phi: \mathbb{N} \to [0, 1]$ is the function describing the probability that an infection kills the specific particle at the moment when it gets infected, with respect to the number of infections the particle has at that moment. Here, we assume that the more infections a node carries simultaneously, the more likely the node is going to die, so $\phi$ is assumed to be non-decreasing, and we assume the function has the property that
\begin{align*}
\phi(0) = 0, \hspace{0.1cm} \phi(k) = 1 \text{ for all } k\geq M,
\end{align*}
where $M\in \mathbb{N}$ is the constant representing the maximum total number of infections that a node can carry at any moment.

Therefore, each virus is healed at rate 1, infects any neighbor at rate $\lambda$ independently, and kills a node following a function of number of infections the node currently carries at the moment when the virus arrives at the node.

We also define the death and survival of the process as following: 
\begin{definition}
The multi-virus contact process is said to \textbf{die out} if 
    \begin{equation}
     P_{\lambda} (\xi_t \ne \emptyset \hspace{0.2cm} \forall t) = 0,
\end{equation}
and to \textbf{survive} if 
\begin{equation}
    P_{\lambda} (\xi_t \ne \emptyset \hspace{0.2cm} \forall t) > 0.
\end{equation}
\end{definition}

By this definition, it is clear that with any particular value of $\lambda$ the process can either die out or survive, but not both. 

We also give definition of phase transition for the multi-virus contact process.

\begin{definition}
The MVCP contact process is said to have \textbf{phase transition} if there exists a critical value $\lambda_c > 0$ such that the process will die out if the infection rate of process is smaller than $\lambda_c$ and will survive if the infection rate is larger than $\lambda_c$.
\end{definition}

\maketitle
\section{Finite Scenario}
We start by proving that in any given finite graph, our contact process dies out almost surely.

\begin{theorem}
\label{th: 1}
Assume the contact process starts on a finite graph $G_0 = (V_0, E_0)$, where $|V_0| < \infty$. Then we have
\begin{equation}
    P_{\lambda} (\xi_t \ne \emptyset \hspace{0.15cm} \forall t) = 0 \quad \forall \lambda > 0.
\end{equation}
\end{theorem}

\begin{remark} Theorem \ref{th: 1} shows a similarity between the MVCP contact process with the classical SIS contact process, that both processes have \textbf{no} phase transition in the finite setting. 
\end{remark}

Before proving this theorem, we first define a helpful concept and then prove an essential lemma. 

\begin{definition}
 Any node $\hat{v}$ in $G_0$ is said to have the immortal property, or to be immortal, if $\forall N \in \mathbb{N}$,  $\exists t' \in [0, \infty)$ such that the total number of infections that $\hat{v}$ experienced before time $t'$ is larger than $N$. 
\end{definition}

\begin{remark}
    It is clear that a node being immortal is a random event, since at time t = 0 one might not know if a node turns out to be immortal.
\end{remark}

\begin{lemma}
\label{le:1}
    No node can have immortal property in any graph $G_0$ almost surely.
\end{lemma}

\begin{proof}
Let $\epsilon > 0$ be given. Choose $N_\epsilon = \bigg \lceil \frac { \log \epsilon}{ \log(1-\phi(1))} + 1 \bigg \rceil$. Choose an arbitrary vertex $\hat{v}$ from $V_0$, then for $\hat{v}$, it has three possible situations: it may never get $N_\epsilon$ infections and survive throughout the time, it may die before or when getting $N_\epsilon$ infections, or it may survive at all times after the $N_\epsilon$ infections at site $\hat{v}$ has occurred. Moreover, the probability for the node to survive after getting $N_\epsilon$ infections is
\begin{equation}
\label{neweq 1}
    \mathbb{P} ( \xi(\hat{v}) \ne \emptyset) = \prod_{i=1}^{N_\epsilon} (1-\phi(k_i)) \leq \prod_{i=1}^{N_\epsilon} (1-\phi(1)) = (1-\phi(1))^{N_{\epsilon}} < \epsilon. 
\end{equation}
where $k_i$ represent the number of infections $\hat{v}$ is carrying at the moment this site gets its $i^{th}$ infection.

The first inequality in \eqref{neweq 1} holds by the monotonicity of the cumulative distribution function, as we have $\phi(x_i) \geq \phi(1)$ for all $x_i \geq 1$. \eqref{neweq 1} shows that for any positive value $\epsilon$, there always exists a $N_\epsilon$ such that the probability for $\hat{v}$ to survive after $N_\epsilon$ infections is smaller than $\epsilon$. Therefore, there exists a $N_0$ such that the probability for $\hat{v}$ to never get $N_0$ infections or to die before or when getting $N_0$ infection converges to 1 if we let $\epsilon$ goes to 0. However, in either case $\hat{v}$ is not immortal. This shows that any vertex cannot be immortal almost surely.
\end{proof}

Now we prove \textbf{Theorem \ref{th: 1}}.

\begin{proof}

It is sufficient to show that for all $\lambda \in \mathbb{R_+}$, the contact process with infection rate $\lambda$ will die out at some time $T$ which is almost surely finite. Take an arbitrary $\lambda \in \mathbb{R_+}$ to be the infection rate.

Due to the fact that no node in $V_0$ is immortal, then for any arbitrary node $v_i \in V_0$, there exists a $N_i^{\lambda} \in \mathbb{N}$ almost surely that the node will either die or never be infected again after getting $N_i^{\lambda}$ infections. Since each infection is healing following a Poisson process with rate 1, if a node survives after its $N_i^\lambda$ infections, the infections on it will eventually be cured as surely, so there exists a finite time $t_i^{\lambda} \in [0, \infty)$ such that $v_i$ will either die or survive forever after $t_i^{\lambda}$ almost surely. Therefore, by taking $T^{\lambda} = \max_{i \leq |V_0|}{t_i}$, all nodes either died or will never be infected again starting from $T^{\lambda}$, and we finish the proof.

\end{proof}

Next we give an example of a finite tree that satisfies Theorem \ref{th: 1}. This example will be helpful in the infinite case.

\maketitle
\subsection{Finite tree with fixed offspring number}
Let $T_{d,n}$ be a labeled finite tree with fixed offspring number $d \geq 2$ and fixed number of vertices $n$. Therefore, the degree of every node is $d + 1$, except for the leaf nodes, each of which has degree 1, and the root of the tree which has degree $d$. We denote the root of the tree as $ \{0 \} $.

If we apply Theorem \ref{th: 1} to this finite tree, we then know that the process will die out on $T_{d,n}$ for all $\lambda$ almost surely. Besides, one can observe from Figure 1 that whenever a site is removed from the graph, the graph will be split into a finite union of finite sub-trees, and in each sub-tree the offspring number of each node within it will be less than or equal to $d$. By the same token, when working with an infinite tree, one can observe that the removal of nodes will transform the whole graph to a finite union of finite and infinite sub-trees. This observation will be useful later on in the infinite tree setting.

\begin{figure}
    \centering
    \label{fig:1}
\begin{tikzpicture}
    \tikzstyle{infected}=[fill={rgb,255: red,191; green,0; blue,64}, draw={rgb,255: red,191; green,0; blue,64}, shape=circle]
    \tikzstyle{Dead}=[fill=blue, draw=blue, shape=circle]
    \tikzstyle{none}=[fill=white, draw=black, shape=circle]
    \node [none] (0) at (0, 0) {};
    \node [none] (1) at (4/3, 6/3) {};
    \node [none] (2) at (5/3, 0) {};
    \node [none] (3) at (4/3, -6/3) {};
    \node [none] (4) at (9/3, 4/3) {};
    \node [none] (5) at (10/3, 6/3) {};
    \node [none] (6) at (9/3, 8/3) {};
    \node [none] (7) at (10/3, 0) {};
    \node [none] (8) at (4/3, 6/3) {};
    \node [none] (9) at (10/3, 1/3) {};
    \node [none] (10) at (10/3, -1/3) {};
    \node [none] (11) at (13/3, 10/3) {};
    \node [none] (12) at (14/3, 9/3) {};
    \node [none] (13) at (13/3, 8/3) {};
    \node [none] (14) at (13/3, 4/3) {};
    \node [none] (15) at (13/3, 5/3) {};
    \node [none] (16) at (13/3, 3/3) {};
    \node [none] (17) at (18/3, 9/3) {};
    \node [none] (18) at (18/3, 10/3) {};
    \node [none] (19) at (18/3, 8/3) {};
    \node [infected] (20) at (0, 0) {};
    \node [infected] (21) at (4/3, 6/3) {};
    \node [infected] (22) at (5/3, 0) {};
    \node [Dead] (23) at (9/3, 8/3) {};
    \node [infected] (24) at (14/3, 9/3) {};
    \node [infected] (25) at (18/3, 10/3) {};
    \node [infected] (26) at (13/3, 8/3) {};
    \node [infected] (27) at (13/3, 3/3) {};
    \draw (1.center) to (0.center);
    \draw (0.center) to (3.center);
    \draw (0.center) to (2.center);
    \draw (1.center) to (5.center);
    \draw (1.center) to (4.center);
    \draw (1.center) to (6.center);
    \draw (2.center) to (9.center);
    \draw (2.center) to (7.center);
    \draw (2.center) to (10.center);
    \draw (6.center) to (13.center);
    \draw (6.center) to (12.center);
    \draw (6.center) to (11.center);
    \draw (4.center) to (15.center);
    \draw (4.center) to (14.center);
    \draw (4.center) to (16.center);
    \draw (12.center) to (18.center);
    \draw (12.center) to (17.center);
    \draw (12.center) to (19.center);
\end{tikzpicture}
    \caption{An example of $T_{3,19}$ at some moment $\tilde{t}$, in which white nodes are normal nodes, red nodes are infected, and the blue node is dead.}
\end{figure}

\maketitle
\section{Infinite Scenario}

\maketitle
\subsection{Infinite tree with fixed degree}
Now we move our focus to infinite trees. 

Let $T_d$ be an infinite tree with fixed degree number $d, d \in \mathbb{N}, d \geq 3$, i.e each node is connected to its parent and at least two offspring. We denote the root of the tree as $ \{0 \} $. Let  $\xi = ( \xi_{t} : t\geq 0) $ be the multi-virus contact process on $T_d$ with infection rate $\lambda$ and healing rate 1, as specified in the introduction, and initial state $\xi_0 = \{ 0 \}$. 

In this case, we \textbf{cannot} assume monotonicity of survival of the process with respect to infection rate $\lambda$. The reason for not assuming monotonicity is clear: with a higher infection rate the virus will spread faster, and therefore further in the graph, but the infected nodes will also get more infections per time. The infected nodes are likely to result in a faster death, which is definitely bad for the survival of virus. Therefore, more work is needed to verify the existence of a critical value of $\lambda$ for the phase transition from death to survival of the process. 

As a matter of fact, the survival of the process may also not be monotonic with respect to the healing rate. When the healing rate increases, infections will be cured faster, but the nodes also are expected to live longer, which could be helpful for the survival of virus. Therefore, more work is needed to prove the existence of a critical value of the healing rate as well.

We can still give lower and upper bounds for $\lambda$, which guarantee the survival or death of the process, respectively.

\begin{theorem}
\label{upperb}
Assume $(1-\phi(1))(d-2) + 1 - M > 0 $, where $M$ is the same constant as in the definition of cumulative distribution function of death, i.e. $\phi(k) = 1 \text{ for all } k\geq M$. Then there exists $\lambda_* > 0$ such that if $\lambda \leq \lambda_*$ the contact process on $d$-regular infinite trees \textbf{dies out} almost surely. Further,
\begin{equation*}
    \displaystyle \lambda_* \geq \frac{1}{ \displaystyle (1-\phi(1))(d-2) + (1-2\phi(2))}.
\end{equation*}
\end{theorem}

\begin{theorem}
\label{thm:3}
Assume $1-\phi(1) - \phi(2) > 0$.  there exists $\lambda^* > 0$ such that if $\lambda \geq \lambda^*$ the contact process \textbf{survive} on $d$-regular infinite trees almost surely. Further, 
\begin{equation*}
    \displaystyle \lambda^* \leq \frac{1}{1-\phi(1)-\phi(2)}.
\end{equation*}
\end{theorem}

\begin{remark}
    The reason behind we need the assumptions in both theorems is clear: we need the denominator to be greater than 0 throughout the proof. Here we give some heuristic explanations about the assumptions we made in both theorems. In Theorem \ref{upperb}, we assume the offspring number of the tree to be larger than the maximum total number of infections of any node, and the ratio between the offspring number and maximum infection number is determined by the function of death. In Theorem \ref{thm:3}, we assume the virus not to be too \textit{detrimental} and do not have high probability of killing the node in the first or second infections.  
\end{remark}

\maketitle
\subsubsection{Proof of Theorem \ref{upperb}}
Let $\rho \in (0, 1)$, and let $\nu_{\rho}(A) = \rho^{I_A}$ for any finite subset $A$ of the vertex-set $V_0$ of $T_d$, where $I_A$ is the total number of infections we have on $A$. At any given moment $t$,  we define $V_t^*$ are the unhealthy but alive vertices in $V_0$ at moment $t$, i.e. $V_t^* = \{v \in V_0, \quad \xi_t(v) \ne 0 \text{ or } \emptyset \}$. It is clear that $V_t^*$ is a finite set. Now we work with the process $\nu_{\rho}(V_t^*)$. Let $g^A_{\lambda}(t) = E^{A}_{\lambda}(\nu_{\rho}(V_t^*))$, which is the expected value of the process $\nu_{\rho}(V_t^*)$ with infection starting on set $A$ with infection rate $\lambda$. 

Before the proof begins, here is a short outline of the reasoning. We first find the lower bound of the infection rate $\lambda$ at which $(g^A_\lambda(t))' \leq 0$. If $ (g^{V_u^*}_\lambda(t))' \leq 0$ under all circumstances because of the $\lambda$ we choose, then by the Markov property, both the conditional probability and expectation do not depend on the current time, so
\begin{equation}
\label{markov}
    \frac{d}{du}{g_{\lambda}^{A}(u)} =E^{A}_{\lambda} \bigg(\frac{d}{dt}{g_{\lambda}^{V_u^*}(t)} \bigg| _{t=0} \bigg) \leq 0,
\end{equation}
implying that $g^A_\lambda(u)$ is non-increasing in $u$. Now, assume we start the infection at the root of the tree, then we have $g^{\{0\}}(0) = \rho < 1$, so therefore $\lim_{u \to \infty} g^A(u) < 1$. On the other hand, note that when all virus disappears, we have $\lim_{t \to \infty} g^A(t) = 1$ representing the death of the process, so when \eqref{markov} holds, the process will survive. The infection rate $\lambda$ must be smaller than the lower bound at which $(g^A_\lambda(t))' \leq 0$ in order for the process to die out, so the lower bound becomes the upper bound in Theorem \ref{upperb}.

In the following proofs, we will discuss three special cases and generalize them to get our final results. 

\begin{lemma}
\label{lem3.1}
Let $A$ be any finite subset of the vertex-set $V_0$ of $T_d$. Let all sites in $A$ have exactly 1 infection at $t=0$. Let $\lambda_1$ be an infection rate which leads to the die out of the process starting at $A$. Then
\end{lemma}

\begin{equation}
      \displaystyle \lambda_1 \leq \frac{1}{ \displaystyle (1-\phi(1))(d-2) + (1-2\phi(2))}.
\end{equation}

\begin{proof}
We have

\begin{equation}
\label{eq:1}
\begin{split}
g^{A}_{\lambda_1}(t) =& |A|t(\frac{\nu_{\rho}(A)}{\rho}) + \lambda_1 N_A t(1-\phi(1))(\nu_{\rho}(A) \rho) \\
&+ \lambda_1|A|t(\frac{\nu_{\rho}(A)}{\rho})\phi(2) +\lambda_1|A|t(\nu_{\rho}(A)\rho)(1-\phi(2)) \\
&+  \nu_{\rho}(A)(1- |A|t -\lambda_1 N_A t(1-\phi(1)) \\
&- \lambda_1|A|t \phi(2)- \lambda_1|A|t(1-\phi(2) ) \\
 &+ o(t),
\end{split}
\end{equation}

as $t \downarrow 0$, where 
\begin{center}
    $N_{A} = |\{ \langle x, y\rangle : x \in A, y\notin A \}|$
\end{center}
is the number of edges of $T_d$ with exactly one end vertex in $A$.

Table 1 shows how each term in \eqref{eq:1} represents a case of interaction between $A$ and its surroundings.

\begin{table}[h!]
  \begin{center}   
    \label{tab:table1}
    \begin{tabular}{||c c||} 
      \textbf{Term} & \textbf{Representation} \\
      \hline
      $|A|t(\frac{\nu_{\rho}(A)}{\rho})$ & Healing \\
      $\lambda_1 N_A t(1-\phi(1))(\nu_{\rho}(A) \rho)$ & Infecting surrounding \\
      $\lambda_1|A|t(\nu_{\rho}(A)\rho)(1-\phi(2))$ & Infecting nodes in $A$ \\
      $\lambda_1|A|t(\frac{\nu_{\rho}(A)}{\rho})\phi(2)$ & Killing the nodes in $A$ \\
       The Rest & No change
    \end{tabular}
     \caption{Representation of Each Term in \eqref{eq:1}.}
  \end{center}
\end{table}

The estimation for the number of neighbors of $A$ can be done in the following way. The minimum number of neighbors of A with given cardinality is achieved when all nodes in $A$ are on the same big branch of the tree. That is, if we remove all other nodes and edges from $G_0$ and only keep the nodes in $A$ and edges connecting them, they still have the shape of one finite tree, and we'll call this tree $T^*_d$ = ($A$, $E^*$). For each node in $T^*_d$, we count its offspring number plus one, which means the overall sum for all node in $A$ is $d|A|$ at the moment $t=0$. Here the 'plus one' represents the node itself. At this moment, this sum includes three types of nodes. The first type is the leaves of leaves in $T^*_d$, which do not belong to $A$ and are counted once. The second type represents the nodes in $A$ besides the root of $T^*_d$, which are counted twice. The third type is the root of $T^*_d$, which belongs to $A$ and is counted once. Now, we want are the neighbors of $A$, $N_A$, which represents the leaves of leaves in $T^*_d$ and the parent of the root of $T^*_d$, so we will have to deduct the size of $A$ minus one twice from the overall sum. That is, 
\begin{equation}
\label{ineq}
    |\{ \langle x, y\rangle : x \in A, y\notin A \}| \geq d |A| - 2(|A|-1),
\end{equation}
and we have
\begin{equation*}
\begin{split}
\frac{d}{dt}{g^{A}_{\lambda_1}(t)} \bigg| _{t=0} =& (1-\rho) \nu_{\rho}(A) \bigg\{ \frac{|A|}{\rho} + (\frac{\lambda_1|A|}{\rho})\phi(2) \\ &- \lambda_1 N_A (1-\phi(1)) - \lambda_1|A|(1-\phi(2)) \bigg \} \\ &\leq (1-\rho) \nu_{\rho}(A) \bigg \{ \frac{|A|}{\rho} + (\frac{\lambda_1|A|}{\rho})\phi(2) \\ &- \lambda_1 \bigg [d |A| - 2(|A|-1) \bigg ]
(1-\phi(1)) - \lambda_1|A|(1-\phi(2)) \bigg \} \\ &= (1-\rho) \nu_{\rho}(A) \bigg\{ |A| \{ \frac{1}{\rho} + \frac{\lambda_1 \phi(2)}{\rho} \\ &-\lambda_1(1-\phi(1))[ d-2 ] -\lambda_1(1-\phi(2)) \} \\ &- 2\lambda_1 (1-\phi(1))  \bigg \} \leq 0
\end{split}
\end{equation*}
when
\begin{equation*}
     \frac{1}{\rho} + \frac{\lambda_1 \phi(2)}{\rho} -\lambda_1(1-\phi(1))[ d-2 ] -\lambda_1(1-\phi(2)) \leq 0,
\end{equation*}
i.e. 
\begin{equation}
    \rho \lambda_1 \bigg \{ (1-\phi(1))(d-2) + (1-\phi(2)) - \frac{\phi(2)}{\rho}  \bigg \}  \geq 1.
\end{equation}

By the argument we have in the outline, in order for the process to die out, we need
\begin{equation}
      \lambda_1 \bigg \{ \rho [(1-\phi(1))(d-2) + (1-\phi(2)) - \frac{\phi(2)}{\rho} ] \bigg \}  \leq 1 \quad \forall \rho \in (0, 1),
\end{equation}
which leads to 
\begin{equation}
    \displaystyle{ \lambda_1   \leq \frac{1}{  (1-\phi(1))(d-2) + (1-2\phi(2))} }
\end{equation}
and finishes the proof.

\end{proof}

\begin{lemma}
\label{lemma2}
Let $A$ be any finite subset of the vertex-set $V_0$ of $T_d$. Let all sites in $A$ have exactly $i$ infections at $t=0$, and $i$ satisfies the condition that
\begin{equation}
    [1-\phi(1)](d-2) + 1 - (i+1)\phi(i+1) > 0. 
\end{equation}
Let $\lambda_i$ be an infection rate which leads to the die out of the process starting at $A$. Then
\begin{equation}
\label{new lem1}
    \displaystyle{ \lambda_i  \leq \frac{1}{ (1-\phi(1))(d-2) + (1-\phi(i+1)) -  i \phi(i+1) }}.
\end{equation}
\end{lemma}

\begin{proof}
    We have

\begin{equation}
\label{eq:2}
\begin{split}
    g^{A}_{\lambda_i}(t) = & i|A|t(\frac{\nu_{\rho}(A)}{\rho}) + \lambda_i i N_A t(1-\phi(1))(\nu_{\rho}(A) \rho) \\ &+ \lambda_i i|A|t (\frac{\nu_{\rho}(A)}{\rho^{i}})\phi(i+1)  + \lambda_i i|A|t(\nu_{\rho}(A)\rho)(1-\phi(i+1)) \\ &+ \nu_{\rho}(A)\bigg [1- i|A|t -\lambda_i i N_A t(1-\phi(1)) \\ &- \lambda_i i|A|t \phi(i+1)- \lambda_i i|A|t(1-\phi(i+1) \bigg]  \\ &+ o(t),
\end{split}
\end{equation}

as $t \downarrow 0$, 

Table 2 shows how each term in \eqref{eq:2} represents a case of interaction between set $A$ and its surrounding.

\begin{table}[h!]
  \begin{center}   
    \label{tab:table2}
    \begin{tabular}{||c c||} 
      \textbf{Term} & \textbf{Representation} \\
      \hline
      $i|A|t(\frac{\nu_{\rho}(A)}{\rho})$ & Healing \\
      $i \lambda_i N_A t(1-\phi(1))(\nu_{\rho}(A) \rho)$ & Infecting surrounding \\
      $i \lambda_i|A|t(\nu_{\rho}(A)\rho)(1-\phi(i+1))$ & Infecting nodes in $A$ \\
      $i \lambda_i|A|t(\frac{\nu_{\rho}(A)}{\rho^i})\phi(i+1)$ & Killing the nodes in $A$ \\
      The Rest & No change
    \end{tabular}
     \caption{Representation of Each Term in \eqref{eq:2}.}
  \end{center}
\end{table}
By taking the first derivative with respect to $t$, we have
\begin{equation*}
\begin{split}
\frac{d}{dt}{g^{A}_{\lambda_i}(t)} \bigg| _{t=0} =& (1-\rho) \nu_{\rho}(A) \bigg\{ \frac{i|A|}{\rho} + (\frac{\lambda_i i|A| (\rho^{i-1} + ... + 1)}{\rho^{i}})\phi(i+1) \\ &- \lambda_i i N_A (1-\phi(1)) - \lambda_i i|A|(1-\phi(i+1)) \bigg \} \\ & \leq (1-\rho) \nu_{\rho}(A) \bigg \{ \frac{i|A|}{\rho} + (\frac{\lambda_i i|A| (\rho^{i-1} + ... + 1)}{\rho^{i}})\phi(i+1) \\ &- \lambda_i i\bigg [(d+1) |A| - 2(|A|-1) \bigg ]
(1-\phi(1)) - \lambda_i i|A|(1-\phi(i+1)) \bigg \} \\ &= (1-\rho) \nu_{\rho}(A) \bigg\{ |A|i \bigg [ \frac{1}{\rho} + \frac{\lambda_i(\rho^{i-1} + ... + 1) \phi(i+1)}{\rho^i} \\ &-\lambda_i(1-\phi(1))[d-2] -\lambda_i(1-\phi(i+1)) \bigg ] - 2\lambda_i i (1-\phi(1))  \bigg \} \\ & \leq 0 
\end{split}
\end{equation*}
when
\begin{equation}
    \frac{1}{\rho} + \frac{\lambda_i(\rho^{i-1} + ... + 1) \phi(i+1)}{\rho^i} -\lambda_i(1-\phi(1))[ d-2 ] -\lambda_i(1-\phi(i+1)) \leq 0,
\end{equation}
i.e. 
\begin{equation}
 \rho \lambda_i \bigg \{ (1-\phi(1))(d-2) + (1-\phi(i+1)) - \frac{\phi(i+1)(\rho^{i-1} + ... + 1)}{\rho^i}  \bigg \}  \geq 1.
\end{equation}
In order for the process to die out, we must have
\begin{equation}
\label{eq:7}
      \lambda_i \bigg \{ \rho  \bigg [ (1-\phi(1))(d-2) + (1-\phi(i+1)) - \frac{\phi(i+1)(\rho^{i-1} + ... + 1)}{\rho^i} \bigg ] \bigg \}  \leq 1  \quad \forall \rho.
\end{equation}
Since the first derivative with respect to $\rho$ inside the big curly bracket is always larger than 0, the maximum is obtained when $\rho = 1$, which further leads to 
\begin{equation}
\label{new lem2}
    \displaystyle{ \lambda_i  \leq \frac{1}{ (1-\phi(1))(d-2) + (1-\phi(i+1)) -  i \phi(i+1) }}
\end{equation}
and completes the proof.
\end{proof}

We can also see that the upper bound on the infection parameter $\lambda_i$ when $i > 1$ in \ref{lem3.1} is strictly larger than the bound we got in the case when $i = 1$ in \eqref{lemma2}, so we can just use the bound for $i = 1$ to be the smallest upper bound when the number of infections are same for the nodes in $A$.

These cases are the general cases which will be helpful when we are dealing with more complicated situations. Now we focus on the case where there are different numbers of infections for the nodes in set $A$. we define $n_0(v)$ to be the number of infections that the node $v \in A$ has at $t=0$. we start with the case where $n_0(v)$ returns 2 values: $i$ and $j$.
\begin{lemma}
    Suppose the set of initial infections $A$ satisfies $A = B \cup C$. Here, each node in $B$ has $i$ infections, and each node in $C$ has $j$ infections at $t=0$. Without the loss of generality, we assume $j > i$. Then the process dies out only when

\begin{equation}
    \displaystyle{ \lambda  \leq \frac{1}{ (1-\phi(1))(d-2) + (1-\phi(i+1)) -  i \phi(i+1) }}.
\end{equation}

\end{lemma}
\begin{proof}
We define $\hat{N}$ to be the number of nearest neighbor pairs $x, y$ with $x \in B$ and $y \in C$,  i.e. $\hat{N} = |\{ (x, y): x \in B, y\in C, x \sim_0 y \}|$, in the graph $G_0$. In this case, the expression for $g^{A}_{\lambda}(t)$ is going to have more terms than the former 2 cases due to the fact that the interaction between $B$ and $C$ should also be involved. So we have
\begin{equation}
\label{eq:3}
\begin{split}
    g^{A}_{\lambda}(t) =& \lambda it(N_B - \hat{N})[1-\phi(1)]\nu_\rho(A)\rho + \lambda i |B|t[1-\phi(i+1)]\nu_\rho(A)\rho \\ &+ \frac{\lambda i |B|t\phi(i+1)\nu_\rho(A)}{\rho^i} + \frac{|B|it\nu_\rho(A)}{\rho} + \lambda jt(N_C - \hat{N})[1-\phi(1)]\nu_\rho(A)\rho \\ &+ \lambda j |C|t[1-\phi(j+1)]\nu_\rho(A)\rho + \frac{\lambda j |C|t\phi(j+1)\nu_\rho(A)}{\rho^j} + \frac{|C|jt\nu_\rho(A)}{\rho}  \\ &+  \lambda i \hat{N} t [1-\phi(j+1)]\nu_\rho(A) \rho + \lambda j \hat{N} t [1-\phi(i+1)]\nu_\rho(A) \rho \\ &+  \frac{\lambda \hat{N} i t \phi(j+1)\nu_\rho(A)}{\rho^j} + \frac{\lambda \hat{N} j t \phi(i+1)\nu_\rho(A)}{\rho^i} \\ &+  \nu_\rho(A) \bigg [1 - \lambda it(N_B - \hat{N})[1-\phi(1)] - \lambda i |B|t[1-\phi(i+1)] \\ &- \lambda i |B|t\phi(i+1) - |B|it - \lambda jt(N_C - \hat{N})[1-\phi(1)] \\ &- \lambda j |C|t[1 -\phi(j+1)] -\lambda j |C|t\phi(j+1) - |C|jt \\ &- \lambda i \hat{N} t [1-\phi(j+1)] - \lambda j \hat{N} t [1-\phi(i+1)] \\ &- \lambda \hat{N} i t \phi(j+1) -  \lambda \hat{N} j t \phi(i+1) \bigg ] \\ &+ o(t) 
\end{split}
\end{equation}
as $t \downarrow 0$, where  
\begin{equation}
\label{eq:BC}
     N_{B} = |\{ \langle x, y\rangle : x \in B, y\notin B \}| 
\end{equation}
and 
\begin{equation}
     N_{C} = |\{ \langle x, y\rangle : x \in C, y\notin C \}| 
\end{equation}
correspond to the respective number of neighbor nodes that set $B$ or set $C$ has.

Table 3 shows how each term in \eqref{eq:3} represents a case of interaction within set $A$ and between $A$ and its surroundings.
\begin{table}[tph]
  \begin{center}   
    \label{tab:table3}
    \begin{tabular}{||c c||} 
      \textbf{Term} & \textbf{Representation} \\
      \hline
      $\lambda it(N_B - \hat{N})[1-\phi(1)]\nu_\rho(A)\rho$ & $B$ infects surrounding \\
      $\lambda i |B|t[1-\phi(i+1)]\nu_\rho(A)\rho$ & $B$ infects itself \\
      $\frac{\lambda i |B|t\phi(i+1)\nu_\rho(A)}{\rho^i}$ & $B$ kills itself \\
      $\frac{|B|it\nu_\rho(A)}{\rho}$ & $B$ heals \\
      $\lambda jt(N_C - \hat{N})[1-\phi(1)]\nu_\rho(A)\rho$ & $C$ infects surrounding \\
      $\lambda j |C|t[1-\phi(j+1)]\nu_\rho(A)\rho$ & $C$ infects itself \\
      $\frac{\lambda j |C|t\phi(j+1)\nu_\rho(A)}{\rho^j}$ & $C$ kills itself \\
      $\frac{|C|jt\nu_\rho(A)}{\rho}$ & $C$ heals \\
      $\lambda i \hat{N} t [1-\phi(j+1)]\nu_\rho(A) \rho$ & $B$ infects $C$ \\
      $\lambda j \hat{N} t [1-\phi(i+1)]\nu_\rho(A) \rho $ & $C$ infects $B$\\
      $\frac{\lambda \hat{N} i t \phi(j+1)\nu_\rho(A)}{\rho^j}$ & $B$ kills $C$ \\
      $\frac{\lambda \hat{N} j t \phi(i+1)\nu_\rho(A)}{\rho^i}$ & $C$ kills $B$ \\
      The Rest & No change
    \end{tabular}
     \caption{Representation of Each Term in \eqref{eq:3}.}
  \end{center}
\end{table}

We take the first derivative of the generating function with respect to $t$ and let $t = 0$. Then we have
\begin{equation}
\label{eq:4}
\begin{split}
\frac{d}{dt}{g^{A}_{\lambda}(t)} \bigg| _{t=0} =&  \lambda i(N_B - \hat{N})[1-\phi(1)]\nu_\rho(A)\rho  + \lambda i |B|[1-\phi(i+1)]\nu_\rho(A)\rho \\ &+ \frac{\lambda i |B|\phi(i+1)\nu_\rho(A)}{\rho^i} + \frac{|B|i\nu_\rho(A)}{\rho} \\ &+ \lambda j(N_C - \hat{N})[1-\phi(1)]\nu_\rho(A)\rho + \lambda j |C|[1-\phi(j+1)]\nu_\rho(A)\rho  \\ &+ \frac{\lambda j |C|\phi(j+1)\nu_\rho(A)}{\rho^j} + \frac{|C|j\nu_\rho(A)}{\rho}  +  \lambda i \hat{N} [1-\phi(j+1)]\nu_\rho(A) \rho \\ &+ \lambda j \hat{N} [1-\phi(i+1)]\nu_\rho(A) \rho  +  \frac{\lambda \hat{N} i \phi(j+1)\nu_\rho(A)}{\rho^j} \\ &+ \frac{\lambda \hat{N} j \phi(i+1)\nu_\rho(A)}{\rho^i} \\ &+ \nu_\rho(A) \bigg [ - \lambda i(N_B - \hat{N})[1-\phi(1)] - \lambda i |B|[1-\phi(i+1)] \\ &- \lambda i |B|\phi(i+1) - |B|i - \lambda j(N_C - \hat{N})[1-\phi(1)] \\ &- \lambda j |C|[1 -\phi(j+1)] - \lambda j |C|\phi(j+1) \\ &- |C|j - \lambda i \hat{N} [1-\phi(j+1)] - \lambda j \hat{N} [1-\phi(i+1)] \\ &- \lambda \hat{N} i \phi(j+1) -  \lambda \hat{N} j \phi(i+1) \bigg ].
\end{split}
\end{equation}
We take the common factor out, so \eqref{eq:4} leads to
\begin{equation}
\label{eq:5}
\begin{split}
    \nu_\rho(A)(1-\rho) & \bigg \{ - \lambda i(N_B - \hat{N})[1-\phi(1)] -\lambda i |B|[1-\phi(i+1)] \\ &+ \frac{\lambda i |B|\phi(i+1)(\rho^{i-1} + ... + 1)}{\rho^i} + \frac{|B|i}{\rho} -\lambda j(N_C - \hat{N})[1-\phi(1)] \\ &-\lambda j |C|[1-\phi(j+1)] +\frac{\lambda j |C|\phi(j+1)(\rho^{j-1} + ... + 1)}{\rho^j} +  \frac{|C|j}{\rho} \\ &-\lambda i \hat{N} [1-\phi(j+1)] - \lambda j \hat{N} [1-\phi(i+1)] \\ &+\frac{\lambda \hat{N} i \phi(j+1)(\rho^{j-1} + ... + 1)}{\rho^j} + \frac{\lambda \hat{N} j \phi(i+1)(\rho^{i-1} + ... + 1)}{\rho^i}    \bigg \} \\ &\leq \nu_\rho(A)(1-\rho) \bigg \{ \frac{|A|j}{\rho} + \frac{\lambda|A|j(\rho^{j-1} + ... + 1)}{\rho^j} \\ &+ \frac{2\lambda\hat{N}j\phi(j+1)(p^{j-1} + ... + 1)}{\rho^j}  - \lambda iN_B[1-\phi(1)] \\ &-\lambda i |B|[1-\phi(i+1)] -\lambda jN_C[1-\phi(1)] -\lambda j |C|[1-\phi(j+1)] \\ &+\lambda i \hat{N} [\phi(j+1)-\phi(1)] + \lambda j \hat{N} [\phi(i+1)-\phi(1)] \bigg \}      
\end{split}
\end{equation}
since $i \leq j$ and $|B| + |C| = |A|$. By \eqref{eq:BC} and \eqref{ineq}
\begin{equation*}
\begin{split}
    N_B \geq [(d+1) |B| - 2(|B|-1)] \\
    N_C \geq [(d+1) |C| - 2(|C|-1)],
\end{split}
\end{equation*}
so we have 
\begin{equation}
\begin{split}
     \eqref{eq:5} \leq& \nu_\rho(A)(1-\rho) \bigg \{ \frac{|A|j}{\rho} + \frac{\lambda|A|j(\rho^{j-1} + ... + 1)}{\rho^j} \\ &+ \frac{2\lambda\hat{N}j\phi(j+1)(p^{j-1} + ... + 1)}{\rho^j} \\ &- \lambda i \bigg [(d+1) |B| - 2(|B|-1) \bigg] [1-\phi(1)] -\lambda i |B|[1-\phi(i+1)] \\ &-\lambda j \bigg [(d+1) |C| - 2(|C|-1) \bigg][1-\phi(1)] -\lambda j |C| [1-\phi(j+1)] \\ &+\lambda i \hat{N} [\phi(j+1)-\phi(1)] + \lambda j \hat{N} [\phi(i+1)-\phi(1)] \bigg \} \leq 0
\end{split}
\end{equation}

when 
\begin{equation}
\begin{split}
    &\frac{|A|j}{\rho} + \frac{\lambda|A|j(\rho^{j-1} + ... + 1)}{\rho^j} + \frac{2\lambda\hat{N}j\phi(j+1)(p^{j-1} + ... + 1)}{\rho^j}  \\ &- \lambda i |B| (d-2) [1-\phi(1)] -\lambda i |B|[1-\phi(i+1)] \\&-\lambda j |C| (d-2)[1-\phi(1)] -\lambda j |C| [1-\phi(j+1)] \\&+\lambda i \hat{N} [\phi(j+1)-\phi(1)] + \lambda j \hat{N} [\phi(i+1)-\phi(1)] \leq 0,
\end{split}
\end{equation}
and this is equivalent to
\begin{equation}
\begin{split}
     &- \frac{\lambda|A|j(\rho^{j-1} + ... + 1)}{\rho^j} - \frac{2\lambda\hat{N}j\phi(j+1)(p^{j-1} + ... + 1)}{\rho^j}  \\ &+ \lambda i |B| (d-2) [1-\phi(1)] + \lambda i |B|[1-\phi(i+1)] +\lambda j |C| (d-2)[1-\phi(1)] \\ &+\lambda j |C| [1-\phi(j+1)] -\lambda i \hat{N} [\phi(j+1)-\phi(1)] \\ &- \lambda j \hat{N} [\phi(i+1)-\phi(1)] \geq \frac{|A|j}{\rho}.
\end{split}
\end{equation}
In order for the contact process to die out, we cannot let \eqref{eq:5} to be smaller than 0, so a necessary condition is 
\begin{equation}
\label{eq:8}
\begin{split}
    &- \frac{\lambda|A|j(\rho^{j-1} + ... + 1)}{\rho^j} - \frac{2\lambda\hat{N}j\phi(j+1)(p^{j-1} + ... + 1)}{\rho^j}  \\ &+ \lambda i |B| (d-2) [1-\phi(1)] + \lambda i |B|[1-\phi(i+1)] \\ &+\lambda j |C| (d-2)[1-\phi(1)] +\lambda j |C| [1-\phi(j+1)] \\ &-\lambda i \hat{N} [\phi(j+1)-\phi(1)] - \lambda j \hat{N} [\phi(i+1)-\phi(1)] \leq \frac{|A|j}{\rho} \quad \forall \rho \in (0, 1).
\end{split}
\end{equation}
We can then split $|A|$ back into $|B|$ and $|C|$. As the terms involving $\hat{N}$ are all negative terms, \eqref{eq:8} becomes true whenever
\begin{equation}
\label{eq:9}
    - \frac{\lambda|B|j(\rho^{j-1} + ... + 1)}{\rho^j} + \lambda i |B| (d-2) [1-\phi(1)] + \lambda i |B|[1-\phi(i+1)] \leq \frac{|B|j}{\rho}
\end{equation}
and 
\begin{equation}
\label{eq:10}
    - \frac{\lambda|C|j(\rho^{j-1} + ... + 1)}{\rho^j} +\lambda j |C| (d-2)[1-\phi(1)] +\lambda j |C| [1-\phi(j+1)] \leq \frac{|C|j}{\rho}.
\end{equation}
Note that both \eqref{eq:9} and \eqref{eq:10} are satisfied when we employ Lemma \ref{lemma2} and use $\lambda_i$ as our sufficient upper bound.
\end{proof}

As the situation with two distinct infection numbers on the initial infectives is proved, we can generalize it and get the following lemma.

\begin{lemma}
\label{lemma3.4}
    Suppose the set of initial infections, $A =\bigcup_{i=1}^n A_i$, where each node in $A_i$
initially has $m_i$ infections. A sufficeint condition for the process to die out is 
    \begin{equation}
    \displaystyle{ \lambda  \leq \frac{1}{ (1-\phi(1))(d-2) + (1-\phi(k+1)) -  k \phi(k+1) }},
\end{equation}
where $k$ is the minimum of $m_i$.
\end{lemma}

Now, one last thing we need to check is whether the statement $(g^A_\lambda(t))' \leq 0$ will still hold when some nodes are killed after a finite time. That is,
\begin{equation}
    E^A_\lambda \bigg(\frac{d}{dt}{g_\lambda^{ \hat{V_u^*}}}(t) \bigg) \bigg| _{t=0} \leq 0,
\end{equation}
where $g_\lambda^{\hat{V_u^*}}(t)$ is the function representing the time when some nodes in $T_d$ are already killed at time $t$, i.e  $\xi(x_i) = \emptyset$ for some $i$. In this scenario, the death of a node will lead to two changes on the graph: removal of the infections on this site and the edges connecting neighbors and itself. We already took care of the decrease in overall infection number when calculating  $g^A_\lambda(t)$, so the only influence brought by killing a node is changing the structure of the graph by removing edges. 

Whenever a node is removed from $G_t$, we can see that $T_d$ will be split into a finite combination of finite graphs and infinite trees. Since we already proved that for any finite graph the process will die out almost surely, we only need to look into 1 infinite tree without loss of generality. In this case, removal of edges will affect our reasoning by reducing the size of $N_A$. If $N_A =  |\{ \langle x, y\rangle : x \in A, y \notin A \}| \leq d |A| - 2(|A|-1)$, this will actually result in an upper bound which is strictly larger than the bound we use for $\lambda_1$, so it will not invalidate our reasoning. For example, if we look at the case in which there is one infection per site, then
\begin{equation}
\begin{split}
    \frac{d}{dt}{g^{A}_{\lambda}(t)} \bigg| _{t=0} =& (1-\rho) \nu_{\rho}(A) \bigg\{ \frac{|A|}{\rho} + (\frac{\lambda|A|}{\rho})\phi(2) - \lambda N_A (1-\phi(1)) \\ &- \lambda|A|(1-\phi(2)) \bigg \} \\ &\leq (1-\rho) \nu_{\rho}(A) \bigg\{ \frac{|A|}{\rho} + (\frac{\lambda|A|}{\rho})\phi(2)  - \lambda|A|(1-\phi(2)) \bigg \} \leq 0
\end{split}
\end{equation}
when 
\begin{equation}
      \rho\lambda \bigg \{ (1-\phi(2)) -\frac{\phi(2)}{\rho} \bigg \} \geq 1.
\end{equation}
If we want the process to die out, we need 
\begin{equation}
      \rho\lambda \bigg \{ (1-\phi(2)) -\frac{\phi(2)}{\rho}) \bigg \} \leq 1 \quad \forall \rho,
\end{equation}
so $\lambda \leq \frac{1}{1-2\phi(2)}$, and we can tell that this bound is strictly larger than the upper bound we use for $\lambda_1$, and thus it erases our concern.  

Together with Lemma \ref{lemma3.4}, we complete the proof of Theorem \ref{upperb}.

\begin{remark}
    It is clear that the method will not work on any infinite graphs which are not regular, as in this case we cannot determine the number of neighbors for each node beforehand. Thus we need new methods to deal with graphs like infinite Erd\H{o}s-R\'enyi random graphs or Preferential Attachment graphs.
\end{remark}
\maketitle
\subsubsection{Proof of Theorem \ref{thm:3}}
We also give a short outline for proving Theorem \ref{thm:3}. We look into the number of infections that the whole graph is carrying throughout the time, and we use an integer-valued jump process to represent it. The jump process will have an absorption site at $0$, representing the death of the process. It will also have some positive probability of increasing and decreasing, representing the change of overall number of infection throughout the time. Now, if we find the biggest value of $\lambda$ which, by changing the probability of jumping accordingly, will lead the jump process to its absorption state almost surely, then we need $\lambda$ to be larger than this value in order for the contact process to survive.

\begin{proof}
Let $(N_t)_{t \geq 0}$ be the integer-valued process representing the number of infections on $V_t$, i.e. $(N_t) = |(\xi_t)| \hspace{0.2cm} \forall t$. 

$N_t$ increases by 1 whenever the virus passes to healthy node or to the sites that are already infected. By the setup of our contact process, a virus is successfully passed to the healthy nodes with rate $\lambda (1-\phi(1))$, and is passed to the nodes that are already infected with rate $\lambda (1-\phi(i+1))$,  where $i$ represents the number of infections the recipient is currently carrying. Since $\phi(x)$ is monotonically increasing with respect to $x$, we have the rate of infection to be smaller than $\lambda (1-\phi(1))$.

Also, $N_t$ decreases by 1 whenever the virus is healed, and when the virus is killing the host by passing to the sites that are already infected, $N_t$ decreases by the number of infection on the host. By the same argument that $\phi(x)$ is monotonically increasing with respect to $x$, we have the rate of decreasing is strictly larger than 1+ $\lambda \phi(2)$.

Now, we couple the $(N_t)_{t \geq 0}$ with a 1 dimension continuous-time random walk $(W_t)_{t \geq > 0}$ on $\mathbb{N}$ including 0, with absorbing site on 0. Note that when the walk achieves the absorbing site, it means that there is no virus on the graph anymore, so the contact process dies out. In this case, $(N_t)$ is stochastically dominated by $(W_t)$ if the probability of increasing by 1 for $(W_t)$ is 

\begin{equation}
  \displaystyle p_{W} =  \frac{\lambda (1-\phi(1))}{1+ \lambda (1-\phi(1)) + \lambda \phi(2)}. 
\end{equation}

Since $(W_t)$ is a 1 dimensional random walk, we have

\begin{equation}
  \displaystyle \lim_{t \to \infty} P(W_t = 0) = 1 \quad \text{ if } \quad p_W \leq \frac{1}{2},
\end{equation}
which is equivalently to say that

\begin{equation*}
  \displaystyle\frac{\lambda (1-\phi(1))}{1+ \lambda (1-\phi(1)) + \lambda \phi(2)}  \leq \frac{1}{2} 
\end{equation*}
\begin{equation}
\label{eq:last}
  \displaystyle \lambda \leq \frac{1}{1-\phi(1)-\phi(2)}.
\end{equation}

Any $\lambda$ smaller than the value in \eqref{eq:last} will lead to the death of the process almost surely, as the number of infection in the graph will converge to 0. Therefore, infection rate needs to be larger than the bound in \eqref{eq:last} in order for the contact process to survive, which finishes the proof of Theorem \ref{thm:3}.

\end{proof}

\section*{Acknowledgement}
The author would like to thank Professor Carl Mueller for his invaluable guidance and support with patience throughout the time.

\printbibliography
\end{document}